\documentclass[12pt,reqno]{amsart}

\usepackage{amssymb}
\usepackage{amsmath}
\usepackage{latexsym}
\usepackage{amsthm}
\usepackage{epsfig}
\usepackage{color}
\usepackage{comment}
\usepackage{amsfonts,amssymb,mathrsfs,amscd}

\usepackage{enumitem}

\theoremstyle{plain}
\newtheorem{theorem}{Theorem}[section]
\newtheorem{proposition}{Proposition}[section]
\newtheorem{corollary}{Corollary}[section]
\newtheorem{lemma}{Lemma}[section]
\newtheorem{remark}{\bf Remark}[section]
\theoremstyle{definition}

\newcommand{\rot}{\mathop\mathrm{rot}}
\newcommand{\grad}{\mathop\mathrm{grad}}
\renewcommand{\div}{\mathop\mathrm{div}}
\newcommand{\Tr}{\mathop\mathrm{Tr}}

\def\({\left(}
\def\){\right)}

\def\Tr{\operatorname{Tr}}
\begin{document}
 \title[Regularized Euler equations on the sphere]
 { Dimension estimates for the attractor  of the regularized damped Euler
 equations on the sphere}
\author[ A. Ilyin,  A. Kostianko, and  S. Zelik] {
Alexei Ilyin${}^1$, Anna Kostianko${}^{3,4}$, and Sergey
Zelik${}^{2,3}$}



\thanks{%This work was supported by Moscow Center for Fundamental and Applied Mathematics,
%Agreement with the Ministry of Science and Higher Education of the Russian Federation,
%No. 075-15-2019-1623 and by the Russian Science Foundation grant No.19-71-30004 (sections 2-4).
The second  author was  partially supported  by the Leverhulme grant No. RPG-2021-072 (United Kingdom).}

\email{ilyin@keldysh.ru}
\email{a.kostianko@imperial.ac.uk}
\email{s.zelik@surrey.ac.uk}
\address{${}^1$ Keldysh Institute of Applied Mathematics, Moscow, Russia}
\address{${}^2$ University of Surrey, Department of Mathematics, Guildford, GU2 7XH, United Kingdom.}
\address{${}^3$ \phantom{e}School of Mathematics and Statistics, Lanzhou University, Lanzhou\\ 730000,
P.R. China}
\address{${}^4$  Imperial College, London SW7 2AZ, United Kingdom.}

\begin{abstract}
We prove  existence of the global attractor of the
damped and driven    Euler--Bardina equations
on the 2D sphere and on arbitrary domains on the sphere
 and give  explicit  estimates of its fractal dimension
in terms of the physical parameters.
\end{abstract}

\maketitle
%\small\tableofcontents

\setcounter{equation}{0}
 \section{Introduction and main result}

The following regularized Euler system has attracted considerable attention
over the last years
\begin{equation}\label{DEalpha-intro}
\left\{
  \begin{array}{ll}
    \partial_t u+(\bar u,\nabla_x)\bar u+\gamma u+\nabla_x p=g,\  \  \\
    \operatorname{div} \bar u=0,\quad u(0)=u_0,\quad  u=(1-\alpha\Delta_x)\bar u.
  \end{array}
\right.
\end{equation}
Here  $g$ is the forcing term, $\gamma u$ is the Ekman damping term that makes the
system dissipative, and  $\alpha>0$ is a small parameter
of dimension $(\textrm{length})^2$ so that
 $\bar u$ is a smoothed vector function with high
 spatial modes filtered out.

In 3D this system is studied as a subgrid scale model of turbulence
and is known in literature  as the inviscid Euler--Bardina model
 \cite{BFR80}. The asymptotic behavior of solutions and estimates for the number of the
degrees of freedom  for this model  and the similar Navier--Stokes--Voight
model were studied in~\cite{Bardina, Titi-Varga} (see also the references therein).

A comprehensive analysis of this system from the
point of view of attractors has recently been done in~\cite{Lap70, arxiv}.
Explicit upper bounds for the fractal dimension of system~\eqref{DEalpha-intro}
on the 2D torus were obtained in~\cite{Lap70} and, furthermore,
by the instability analysis on the corresponding generalized Kolmogorov flows
it was also shown there that the upper estimates are optimal in the limit $\alpha\to0$.

A more difficult 3D case was studied in detail in \cite{arxiv} for system
\eqref{DEalpha-intro} on the 3D torus and in a domain $\Omega\subseteq\mathbb{R}^3$,
where we again obtained explicit upper bounds for the attractor dimension and
the estimate for the 3D torus is also optimal as $\alpha\to0$.

Motivated by possible geophysical applications we study in this work
the regularized Euler--Bardina system on the 2D sphere and in proper
domains $\Omega$ on the sphere:
\begin{equation}\label{DEalpha}
\left\{
  \begin{array}{ll}
    \partial_t u+\nabla_{\bar u}\bar u+\gamma u+\nabla p=g,\  \  \\
    \operatorname{div} \bar u=0,\quad u(0)=u_0,
    \quad u=(1-\alpha\mathbf{\Delta})\bar u.
  \end{array}
\right.
\end{equation}
The corresponding phase space  with respect to
$\bar u$ is
\begin{equation}\label{H1}
\bar u\in\mathbf{H}^1:=\left\{
                       \begin{array}{ll}
                         \dot{\mathbf{H}}^1_0(\Omega),
                         & \Omega\varsubsetneq\mathbb{S}^2, \\
                         \mathbf{H}^1(\mathbb{S}^2), & \Omega=\mathbb{S}^2
                       \end{array}
                     \right.\qquad\operatorname{div}\bar u=0,
\end{equation}
and $\bar u=(1-\alpha\mathbf{\Delta})^{-1}u$ in the case of $\mathbb{S}^2$, while
for $\Omega\varsubsetneq\mathbb{S}^2$ we recover $\bar u$ by solving the Stokes problem
in $\Omega$:
$$
\bar u-\alpha\mathbf{\Delta}\bar u+\nabla q=u,\quad \div \bar u=0,\quad
\bar u\vert_{\partial\Omega}=0.
$$

In \eqref{DEalpha} $\nabla_uu$ is the covariant derivative of $u$
along $u$ for which we have~\cite{I90}
$$
\nabla_uu=\nabla\frac{u^2}2-u^\perp\rot u.
$$
In the vector case by the  Laplace operator acting on (tangent)
vector fields on $\mathbb{S}^2$ we mean  the Laplace--de Rham
operator $-d\delta-\delta d$ identifying $1$-forms and vectors.
Then for a two-dimensional manifold
 we have \cite{I90}
\begin{equation}\label{vecLap}
\mathbf{\Delta} u=\nabla\div u-\rot\rot u,
\end{equation}
where the  operators $\nabla=\grad$ and $\div$ have the
conventional meaning. The operator $\rot$ of a vector $u$ is a
scalar  and for a scalar $\psi$, $\rot\psi$ is a vector: $\rot
u:=\div(u^\perp)$, $\rot\psi:=\nabla^\perp\psi$, where  in the
local frame $u^\perp=(u_2,-u_1)$, that is, $\pi/2$ clockwise rotation
of $u$ in the local tangent plane. Integrating by parts we obtain
\begin{equation}\label{byparts}
(-\mathbf{\Delta} u,u)=\|\rot u\|^2_{L^2}+\|\div u\|^2_{L^2}.
\end{equation}

We can now state the main result of this work proved in
Section~\ref{S:bounds}.

 \begin{theorem}\label{Th:into-main}
Let $\Omega\subseteq\mathbb{S}^2$. The  regularized  Euler system
\eqref{DEalpha} has a global attractor  $\mathscr A$ in $\mathbf{H}^1$ with finite fractal
dimension satisfying the following upper bound
\begin{equation}\label{upper-main}
\dim_F{\mathscr A}\le\frac1{8\pi}\cdot\left\{
\aligned
&\frac1{\alpha\gamma^4}
\min\left(\|\operatorname{rot}g\|^2_{L^2},\ \frac{\|g\|^2_{L^2}}{2\alpha}\right),
\quad \Omega=\mathbb{S}^2\\
&\frac{\|g\|^2_{L^2}}{2\alpha^2\gamma^4},\quad \Omega\varsubsetneq\mathbb{S}^2.
\endaligned
\right.
\end{equation}
\end{theorem}

In Section~\ref{S:attr} we prove dissipative estimates, write the system as an ODE in
$\mathbf{H}^1$ with  bounded nonlinearity and construct the global attractor.
In the Appendix (which is of an independent interest)
we prove in the spirit of \cite{LiebJFA}  collective Sobolev inequalities
on the sphere for families of functions with orthonormal derivatives.

We point out in conclusion that estimates~\eqref{upper-main} are \emph{exactly}
the same as those in the case of the 2D torus $\mathbb{T}^2$ (and $\mathbb{R}^2$),
and $\Omega\varsubsetneq\mathbb{R}^2$, see, respectively~\cite{Lap70,arxiv}. Furthermore,
since as shown in \cite{Lap70} the estimate on $\mathbb{T}^2$ is optimal, we have a strong
evidence that it is also true for $\mathbb{S}^2$.

\setcounter{equation}{0}
\section{A priori estimates and the global attractor}\label{S:attr}

Before we prove two types of energy estimates
we recall the following two orthogonality relations~\cite{I90}.

\begin{lemma}\label{L:orth}
Let $\Omega\subseteq\mathbb{S}^2$ and
let $u,v\in T\mathbb{S}^2$, $\div u=0$ be smooth vector functions in $\Omega$.
In case when $\Omega\subsetneq\mathbb{S}^2$ we further suppose that
$u\vert_{\partial\Omega}=0$.
Then
\begin{equation}\label{S2orth}
\aligned
\int_\Omega(\nabla_uv\cdot v)dS=0,\quad
\int_{\mathbb{S}^2}(\nabla_uu\cdot \mathbf{\Delta}u)dS=0.
\endaligned
\end{equation}
\end{lemma}
\begin{proof}
To prove the first identity we use the following result
in differential geometry, see, for instance, \cite{Dubrovin}.
Namely, suppose that  $M$ is a surface in $\mathbb{R}^d$ and let
$u,v$ be tangent vector functions on $M$. Let $u$ and $v$ be
somehow prolonged  in  a neighborhood of $M$
in $\mathbb{R}^d$ with fixed Cartesian system. We denote them as
$\tilde u$, $\tilde v$.
Then at a point $x\in M$
$$
\nabla_uv(x)=\pi \left(\sum_{i=1}^d\tilde u_i\partial_i\tilde v_j\right),
$$
where $\pi$ is the projection on the tangent plane to $M$ at $x$.

Using this we obtain taking into account that $u$ is tangent
$$
\aligned
\int_\Omega(\nabla_uv\cdot v)dS=
\int_\Omega\sum_{i,j=1}^3\tilde u_i\partial_i\tilde v_j\tilde v_jdS=
\frac12\int_\Omega\sum_{i,j=1}^3\tilde u_i\partial_i(\tilde v_j)^2dS=\\=
\frac12\int_\Omega\sum_{j=1}^3u\cdot\nabla(v_j)^2dS=
-\frac12\int_\Omega\sum_{j=1}^3 (v_j)^2\div u\,dS=0.
\endaligned
$$
We point out that $\nabla$ and $\div$ here are the
\emph{surface} gradient and divergence.

The second identity follows since setting $\omega=\rot u$  we have
$$
\int_{\mathbb{S}^2}\nabla_uu\cdot \mathbf{\Delta}udS=
\int_{\mathbb{S}^2}\left(\frac12\nabla u^2-\omega u^\perp\right)
\cdot\nabla^\perp\omega dS=-
\int_{\mathbb{S}^2}\frac12u\cdot\nabla(\omega)^2dS=0.
$$
\end{proof}

\begin{proposition}\label{Prop:2.1} Let $u$ be a smooth solution of equation \eqref{DEalpha}.
Then for any $\Omega\subseteq\mathbb{S}^2$ the following dissipative energy estimate holds:
%$$
\begin{equation}\label{energy}
\|\bar u(t)\|^2_\alpha\le \|\bar u(0)\|^2_\alpha e^{-\gamma t}+\frac1{\gamma^2}\|g\|^2_{L^2},
\end{equation}
%$$
where
%$$
$$%\begin{equation}\label{2.norm}
\|\bar u\|^2_\alpha:=\|\bar u\|^2_{L^2}+\alpha\|\rot\bar u\|^2_{L^2}.
$$%\end{equation}
%$$
For $\Omega=\mathbb{S}^2$ estimate \eqref{energy} still holds and,
in addition,
%$$
\begin{equation}\label{enstrophy}
\|\bar \omega(t)\|^2_\alpha\le \|\bar \omega(0)\|^2_\alpha e^{-\gamma t}+
\frac1{\gamma^2}\|\rot g\|^ 2_{L^2},
\end{equation}
%$$
%$$
where $\omega=\rot u$, $\bar\omega=\rot \bar u$,
$\omega=(1-\alpha\Delta)\bar\omega$  and
$$%\begin{equation}\label{2.norm-ens}
\|\bar \omega\|^2_\alpha:=\|\bar \omega\|^2_{L^2}+\alpha\|\nabla
\bar \omega\|^2_{L^2}.
$$%\end{equation}
%$$
\end{proposition}
\begin{proof} Indeed, taking the scalar product of \eqref{DEalpha} and  $\bar u$,
 using \eqref{byparts}, that is,
$$
(u,\bar u)=\|\bar u\|_{L^2}^2+\|\rot\bar u\|_{L^2}^2
$$
and the first identity in~\eqref{S2orth}, we obtain
%$$
\begin{multline}\label{2.en-est}
\frac d{dt}\left(\|\bar u\|^2_{L^2}+\alpha\|\rot\bar u\|^2_{L^2}\right)+
2\gamma\left(\|\bar u\|^2_{L^2}+\alpha\|\rot\bar u\|^2_{L^2}\right)=2(g,\bar u)\le\\\le
2\|g\|_{L^2}\|\bar u\|_{L^2}\le
\gamma\|\bar u\|^2_{L^2}+\frac1\gamma\|g\|_{L^2}^2.
\end{multline}
%$$
Applying the Gronwall inequality, we obtain estimate \eqref{energy}.

The proof of~\eqref{enstrophy} is similar, and we take the scalar
product of \eqref{DEalpha} and  $-\mathbf{\Delta}\bar u$ instead.
Using the second orthogonality relation  we obtain
\begin{multline*}%\label{enstr-est}
\frac d{dt}\left(\|\bar \omega\|^2_{L^2}+\alpha\|\nabla\bar \omega\|^2_{L^2}\right)+
2\gamma\left(\|\bar \omega\|^2_{L^2}+\alpha\|\nabla\bar\omega\|^2_{L^2}\right)=
2(g,\rot\rot\bar u)=\\=2(\rot g,\bar \omega)
\le
\gamma\|\bar \omega\|^2_{L^2}+\frac1\gamma\|\rot g\|_{L^2}^2
\end{multline*}
and complete the proof as before.
\end{proof}

The following time averaged estimates are essential in
Section~\ref{S:bounds}. %the estimates of the attractor dimension.

\begin{corollary}
%Let $u$ be a smooth solution of equation \eqref{DEalpha}.
For any $\Omega\varsubsetneq\mathbb{S}^2$
\begin{equation}\label{timenotS2}
\limsup_{t\to\infty}\frac1t\int_0^t\|\rot \bar u(s)\|_{L^2}\,ds\le \frac1{\gamma\sqrt{2\alpha}}\|g\|_{L^2}.
\end{equation}
For $\Omega=\mathbb{S}^2$ estimate \eqref{timenotS2} still holds and,
in addition,
\begin{equation}\label{timeS2}
\limsup_{t\to\infty}\frac1t\int_0^t\|\rot \bar u(s)\|_{L^2}\,ds\le
\frac1{\gamma}\|\rot g\|_{L^2}.
\end{equation}
\end{corollary}
\begin{proof}
Estimate~\eqref{timeS2} immediately follows from~\eqref{enstrophy}.
To see that~\eqref{timenotS2} holds we integrate \eqref{2.en-est}
from $0$ to $t$, divide by $t$ and let $t\to\infty$. Since in view
of~\eqref{energy}, $\|\bar u(t)\|_\alpha$ is bounded, we obtain that
$$
\limsup_{t\to\infty}\frac1t\int_0^t\|\rot \bar u(s)\|^2_{L^2}\,ds\le \frac1{2\alpha\gamma^2}\|g\|^2_{L^2}.
$$
Using H\"older inequality
$$
\frac1t\int_0^t\|\rot\bar u(s)\|_{L^2}\,ds\le
 \(\frac1t\int_0^t\|\rot\bar u(s)\|^2_{L^2}\,ds\)^{1/2},
$$
we obtain~\eqref{timenotS2}.
\end{proof}

We now write equation \eqref{DEalpha}
as  an ODE in a Hilbert space with
bounded nonlinearity. Applying to \eqref{DEalpha}  the
operator
$$
A_\alpha:=(1-\alpha A)^{-1}\Pi
$$
where $A=\Pi\mathbf{\Delta}$ is the Stokes operator in $\Omega$ and $\Pi$ is
the Helmholtz--Leray projection, we obtain

%$$
\begin{equation}\label{1.ODE}
\partial_t\bar u+\gamma \bar u+B(\bar u,\bar u)=\bar g, \ \bar u\big|_{t=0}=\bar u_0,
 \end{equation}
%$$
where $B(\bar u,\bar v):=  A_{\alpha}\Pi\(\nabla_{\bar u}\bar v\)$,
$\bar g=(1-\alpha A)^{-1}\Pi g$.

Arguing as  in \cite{Lap70} and also using the  elliptic regularity for the
Stokes operator we see that $B$ is bounded from $\mathbf{H}^1$
to $\mathbf{H}^{2-\varepsilon}$, $\varepsilon>0$.

As a result, we have written  \eqref{DEalpha} as an
 an ODE in ${\bf H}^1$
with bounded nonlineariry. Therefore the local existence
and uniqueness of a solution as well as (an infinite) differentiability of the
corresponding local solution semigroup are straightforward corollaries of the
Banach contraction principle. The global existence  follows from   a priori estimates
obtained above, so that
we have proved the following theorem.

    \begin{theorem}\label{Th:exist} Let $\bar u_0\in {\bf H}^1(\Omega)$.
     Then there exists a unique global solution $\bar u\in C([0,\infty),{\bf H}^1)$
    of problem \eqref{1.ODE} (which is simultaneously the unique solution of
    \eqref{DEalpha}). In other words, a dissipative solution semigroup
$$
S(t):\mathbf{H}^1\to\mathbf{H}^1,\quad S(t)\bar u_0:= \bar u(t),\ \ t\ge0
$$
 is well defined. Moreover, $S(t)$ is
$C^\infty$-differentiable for every fixed~$t$.
\end{theorem}

Concluding this section we construct the main object of our interest,
namely, the global attractor of the solution semigroup $S(t)$.

\begin{theorem}\label{Th:attr} The semigroup $S(t)$ has a global
attractor in $mathscr{A}\subset\mathbf{H}^1$ which, by definition, is a set
that is\newline
1) compact in $\mathbf{H}^1$, $\mathscr{A}\Subset\mathbf{H}^1$;\newline
2) strictly invariant
$S(t)\mathscr{A}=\mathscr{A}$;\newline
3) attracts bounded sets in
$\mathbf{H}^1$: for every bounded set $B\subset \mathbf{H}^1$ and
every neighborhood $\mathcal{O}(\mathscr{A})$
$$
S(t)B\subset \mathcal{O}(\mathscr{A})\ \text{for}\ t\ge T(B,\mathcal{O}(\mathscr{A})).
$$
\end{theorem}
\begin{proof}
The semigroup $S(t)$ is continuous and dissipative in view of~\eqref{energy}.
To apply a general and by now standard result on the existence of the attractor,
see, for instance, \cite{B-V,T} we only need to establish the asymptotic compactness
of $S(t)$. This is also achieved by the standard splitting of $S(t)$ into a exponentially
decaying part and a uniformly compact part:
$$
S(t)=\Sigma(t)+S_2(t),\quad S_2(t)=S(t)-\Sigma(t),
$$
where
$v(t)=\Sigma(t)\bar u_0$ is the decaying solution
of the linear equation
$$
\partial_t v+\gamma v=0, \ v(0)=\bar u_0,
$$
and $w(t)=S_2(t)\bar u_0$ is the solution of the equation
$$
\partial_tw+\gamma w=G(t):=-B(\bar u,\bar u)+\bar g,\ w(0)=0
$$
with zero initial condition and right-hand side  uniformly bounded in
$\mathbf{H}^{2-\varepsilon}$. Therefore $w$ is uniformly bounded in
$\mathbf{H}^{2-\varepsilon}$ and since $\bar u=v+w$ the asymptotic compactness of
the semigroup $S(t)$ is established.
\end{proof}

\setcounter{equation}{0}
\section{Upper bound for the dimension of the attractor}\label{S:bounds}

\begin{proof}[Proof of Theorem~\ref{Th:into-main}] The solution semigroup   $S(t): {\bf H}^1\to
{\bf H}^1$ is differentiable  with respect to the initial data %(see Corollary \ref{Cor1.sem}),
 so we only need to estimate the global Lyapunov
 exponents for the linearization of equation
\eqref{1.ODE}
on the trajectories lying on the attractor. The linearized system is:
$$
 \begin{cases}
 \partial_t \bar\theta=-\gamma\bar\theta-B(\bar u(t),\bar\theta)-B(\bar\theta,\bar u(t)) =:L_{u(t)}\bar \theta,\\
\div \bar\theta=0,\  \bar\theta\big|_{t=0}=\bar\theta_0\in \mathbf H^1(\Omega),
\end{cases}
$$
where  $B(\bar u,\bar v):=(1-\alpha A)^{-1}\Pi\(\nabla_{\bar u}\bar v\)$.
%In order to
% utilize the  well-known cancelation property
%$$
%((\bar u,\Nx)\bar \theta,\bar\theta)\equiv 0
%$$
%for the inertial term in the Navier-Stokes equations,
It is convenient to define the  scalar product in  ${\bf H}^1$
induced  by the operator
$1-\alpha A$, namely,
%$$
\begin{equation}\label{scal-alpha}
(\bar\theta,\bar\xi)_\alpha=(\bar\theta,\bar\xi)+\alpha(\rot\bar\theta,\rot\bar\xi)=((1-\alpha A)\bar\theta,\bar\xi)
\end{equation}
%$$
 Then, using that $\Pi A_\alpha=A_\alpha$ and $\Pi\bar\theta=\bar\theta$, we obtain
%$$
\begin{equation}\label{=0}
\aligned
(B(\bar u,\bar\theta),\bar\theta)_\alpha=
\left((1-\alpha A)^{-1} \Pi\nabla_{\bar u}\,\bar\theta,(1-\alpha\Delta_x)\bar\theta\right)=\\=
\left((1-\alpha A)^{-1} \Pi\nabla_{\bar u}\,\bar\theta,(1-\alpha\Pi\Delta_x)\bar\theta\right)=\\=
\left(\Pi\nabla_{\bar u}\,\bar\theta,\bar\theta\right)=
(\nabla_{\bar u}\,\bar\theta,\bar\theta)\equiv0.
\endaligned
\end{equation}
%$$

Following the general strategy, see e.g. \cite{T}, the sums of the first $n$
global Lyapunov exponents, which control the dimension, can be estimated from
above by the following numbers:
$$
q(n):=\limsup_{t\to\infty}\sup_{u(t)\in\mathscr A}\sup_{\{\bar\theta_j\}_{j=1}^n}\frac1t\int_0^t
\sum_{j=1}^n(L_{u(\tau)}\bar\theta_j,\bar\theta_j)_\alpha d\tau,
$$
where the first (inner) supremum is taken over all
orthonormal families $\{\bar\theta_j\}_{j=1}^n$ with respect to the
scalar product \eqref{scal-alpha} in ${\bf H}^1$
and the second (middle) supremum  is taken over all   trajectories $u(t)$
on the attractor $\mathscr A$. Then,
using \eqref{=0} %together with the pointwise estimate \eqref{dpointwise}
%proved in Appendix \ref{sB},
we obtain
$$
\aligned
\sum_{j=1}^n(L_{u(t)}\bar\theta_j,\bar\theta_j)_\alpha&=
-\sum_{j=1}^n\gamma\|\bar\theta_j\|^2_{\alpha}-
\sum_{j=1}^n\int_\Omega(\nabla_{\bar\theta_j}\bar u\cdot\bar\theta_j)dS\le\\&\le
 -\gamma n+\frac1{\sqrt{2}}\|\rot \bar u(t)\|_{L^2}\|\rho\|_{L^2},
\endaligned
$$
 where
$$
\rho(s)=\sum_{j=1}^n|\bar\theta_j(s)|^2,
$$
and where we used the following inequality special for the spherical geometry
(and similar to that in the 2D flat case), see~\cite[Lemma 3.2]{IMT}:
$$
\sum_{j=1}^n\int_{\Omega}
(\nabla_{v_j}u(s)\cdot v_j(s))dS\le
2^{-1/2}\|\rho\|_{L^2}\|\rot u\|_{L^2},\ \rho(s)=\sum_{j=1}^n|v_j(s)|^2,  \ \div u=0.
$$
We now use  estimate
 \eqref{S2alpha} from the  Appendix
and obtain
$$
\sum_{j=1}^n(L_{u(t)}\bar\theta_j,\bar\theta_j)_\alpha\le-\gamma n+
\frac1{2\sqrt {2\pi}}\frac{n^{1/2}}{\alpha^{1/2}}\| \rot\bar u(t)\|_{L^2}.
$$
Finally, using \eqref{timenotS2} and  \eqref{timeS2}
we arrive at
$$
q(n)\le-\gamma n +\frac{1}{2\sqrt{2\pi}}\frac{n^{1/2}}{{\alpha}^{1/2}}\cdot
\left\{
\aligned
&\frac1{\gamma}
\min\left(\|\rot g\|_{L^2},\ \frac{\|g\|_{L^2}}{\sqrt{2\alpha}}\right),
\quad \Omega=\mathbb{S}^2,\\
&\frac{\|g\|_{L^2}}{\gamma\sqrt{2\alpha}},\quad \Omega\varsubsetneq\mathbb{S}^2.
\endaligned
\right.
$$
It only remains to recall that, according to the general theory,
any number $n^*$  for which $q(n^*)<0$ is  an  upper bound both for
the Hausdorff \cite{B-V,T} and the fractal \cite{Ch-I2001,Ch-I}
dimension of the global attractor $\mathscr A$. This gives
estimate~\eqref{upper-main}.
%and completes the proof of the theorem.
\end{proof}

\setcounter{equation}{0}
\section{Appendix}\label{S:app}

In this  section we prove the following result that makes it
possible to write the estimates for the dimension of the attractor
in the explicit form.

\begin{theorem}\label{Th:A-Lieb-S2}
Let $\Omega\subseteq\mathbb{S}^2$ be a (curved) domain on
$\mathbb{S}^2$. Let a family  $\{v_j\}_{j=1}^n\in
\mathbf{H}^1_0(\Omega)$, $\div v_j=0$,  be  orthonormal in
$\mathbf{H}^1_0(\Omega)$ with respect to the scalar product
\begin{equation}\label{orth-m}
m^2(v_i,v_j)_{L^2}+(\rot v_i,\rot v_j)_{L^2}=\delta_{ij}.
\end{equation}
Then the function $\rho(s):=\sum_{j=1}^n|v_j(s)|^2$ satisfies the
inequality
\begin{equation}\label{LiebS2}
\|\rho\|_{L^2}\le\frac1{2\sqrt{\pi}}m^{-1}{n^{1/2}}.
\end{equation}
\end{theorem}
\begin{proof}
 We  first recall the basic facts concerning the spectrum of the
scalar Laplace operator $\Delta=\div\nabla$ on the sphere
$\mathbb{S}^{2}$:
\begin{equation}\label{harmonics}
-\Delta Y_n^k=n(n+1) Y_n^k,\quad
k=1,\dots,2n+1,\quad n=0,1,2,\dots.
\end{equation}
Here the $Y_n^k$ are the orthonormal real-valued spherical
harmonics and each eigenvalue $\Lambda_n:=n(n+1)$ has multiplicity $2n+1$.

The following identity is essential in what
follows: for any $s\in\mathbb{S}^{2}$
\begin{equation}\label{identity}
\sum_{k=1}^{2n+1}Y_n^k(s)^2=\frac{2n+1}{4\pi}.
\end{equation}
In the vector case  identity~\eqref{identity} is replaced by its
vector analogue:
\begin{equation}\label{identity-vec}
\sum_{k=1}^{2n+1}|\nabla Y_n^k(s)|^2=n(n+1)\frac{2n+1}{4\pi}.
\end{equation}

Turning to the proof we  first consider the  whole sphere
$\Omega=\mathbb{S}^2$.  Corresponding to the eigenvalue $\Lambda_n=n(n+1)$,
where $n=1,2,\dots$, there is a family of $2n+1$ orthonormal
vector-valued  eigenfunctions  $w_n^k(s)$ of the vector Laplacian
on the invariant space of divergence free
vector-functions, that is, the Stokes operator on $\mathbb{S}^2$
\begin{equation}\label{bases}
\aligned
w_n^k(s)&=(n(n+1))^{-1/2}\,\nabla^\perp Y_n^k(s),\ -\mathbf{\Delta}w_n^k=n(n+1)w_n^k,\
\div w_n^k=0;
\endaligned
\end{equation}
where $k=1,\dots,2n+1$,  and~\eqref{identity-vec} implies the
following identity:% for any $s\in\mathbb{S}^2$ and $n=1,2,\dots,$
\begin{equation}\label{id-vec}
\sum_{k=1}^{2n+1}|w_n^k(s)|^2=\frac{2n+1}{4\pi}.
\end{equation}
%We finally observe that  $-\mathbf{\Delta}$ is strictly positive
%$-\mathbf{\Delta}\ge \Lambda_1I=2I.$
Let us define two operators
$$
 \mathbb{H}= V^{1/2}(m^2-\mathbf{\Delta})^{-1/2}\Pi,\quad  \mathbb{H}^*=\Pi(m^2-\mathbf{\Delta})^{-1/2}V^{1/2}
$$
acting in $T\mathbb{S}^2$, where $V\in L^1\cap L^\infty$ is a
non-negative scalar function  and $\Pi$ is the Helmholtz--Leray
projection.  Then ${\bf K}=
\mathbb{H}^*\mathbb{H}$ is a  compact self-adjoint operator acting
from $\mathbf{L}^2({\mathbb{S}}^2)$ to $\mathbf{L}^2({\mathbb{S}}^2)$ and
$$
\aligned
\Tr \mathbf{K}^2=&\Tr\left(\Pi(m^2-\mathbf{\Delta})^{-1/2}V(m^2-\mathbf{\Delta})^{-1/2}\Pi\right)^2\le\\&\le
\Tr\left(\Pi(m^2-\mathbf{\Delta})^{-1}V^2(m^2-\mathbf{\Delta})^{-1}\Pi\right)=
\Tr\left(V^2(m^2-\mathbf{\Delta})^{-2}\Pi\right),
\endaligned
$$
where we used
 the Araki--Lieb--Thirring inequality for traces \cite{Araki, LT, traceSimon}:
$$
\Tr(BA^2B)^p\le\Tr(B^pA^{2p}B^p),\quad p\ge1,
$$
and the cyclicity property of the trace together with the facts
that $\Pi$ commutes with the Laplacian and that $\Pi$ is a
projection: $\Pi^2=\Pi$. Using the basis of orthonormal
eigenfunctions of the Laplacian \eqref{bases} along with \eqref{id-vec}
in view of the key estimate~\eqref{ineqS2} proved  below we
find that
$$
\aligned
\operatorname{Tr} \mathbf{K}^2\le&
\operatorname{Tr}\left(V^2(m^2-\mathbf{\Delta})^{-2}\Pi\right)=\\=&
\frac1{4\pi}\sum_{n=1}^\infty\frac{2n+1}{\bigl(m^2+n(n+1)\bigr)^2}
\int_{\mathbb{S}^2}V^2(s)dS\le\frac1{4\pi}\frac1{m^2}\|V\|_{L^2}^2.
\endaligned
$$
We can now argue as in~\cite{LiebJFA}. We observe that
$$
\int_{\mathbb{S}^2}\rho(s)V(s)dS=\sum_{i=1}^n\|\mathbb{H}\psi_i\|^2_{L^2},
$$
where
$$
\psi_j=(m^2-\mathbf{\Delta})^{1/2}v_j,\quad j=1,\dots,n.
$$
Next, in view of \eqref{orth-m} the $\psi_j$'s are
orthonormal in $L^2$
and in view of  the variational
principle
$$
\sum_{i=1}^n\|\mathbb{H}\psi_i\|^2_{L^2}=\sum_{i=1}^n(\mathbf{K}\psi_i,\psi_i)
\le\sum_{i=1}^n\lambda_i,
$$
where $\lambda_i$ are the eigenvalues of the
operator $\mathbf{K}$. Therefore
$$
\aligned
\int_{\mathbb{S}^2}\rho(s)V(s)dS&\le\sum_{i=1}^n\lambda_i\le
n^{1/2}\left(\sum_{i=1}^n\lambda_i^2\right)^{1/2}\le\\&\le
n^{1/2}\left(\operatorname{Tr} \mathbf{K}^2\right)^{1/2}\le
\frac{n^{1/2}m^{-1}}{2\sqrt{\pi}}\|V\|_{L^2}.
\endaligned
$$
Setting $V(x):=\rho(x)$ we complete the proof of
\eqref{LiebS2} for $\Omega=\mathbb{S}^2$.

Finally, if $\Omega\varsubsetneq\mathbb{S}^2$ is a proper domain on $\mathbb{S}^2$,
we extend by zero the vector functions $v_j$ outside
$\Omega$ and denote the results by $\widetilde{v}_j$,
so that $\widetilde{v}_j\in {\bf H}^1(\mathbb{S}^2)$
and $\operatorname{div}\widetilde{v}_j=0$. We
further set
$\widetilde\rho(x):=\sum_{j=1}^n|\widetilde{v}_j(x)|^2$.
Then setting $\widetilde{\psi}_i:=(m^2-\mathbf{\Delta})^{1/2}\widetilde{v}_i$,
we see that the system $\{\widetilde\psi_j\}_{j=1}^n$ is orthonormal in
$L^2(\mathbb{S}^2)$ and
$\operatorname{div}\widetilde\psi_j=0$.
Since clearly $\|\widetilde\rho\|_{L^2(\mathbb{S}^2)}=\|\rho\|_{L^2(\Omega)}$,
the proof of the  estimate \eqref{LiebS2} reduces to the case of the whole sphere
and therefore is complete.
\end{proof}

A word for word translation of the above proof gives
a similar inequality in the scalar case, the only difference being
that we have to consider the case of the whole sphere and impose the
zero mean condition.
\begin{theorem}\label{Th:A-Lieb-S2-scal}
 Let a family  $\{\varphi_j\}_{j=1}^n\in
{H}^1(\mathbb{S}^2)$, $\int_{\mathbb{S}^2}\varphi_j(s)dS=0$ be  orthonormal
with respect to the scalar product
$$
m^2(\varphi_i,\varphi_j)_{L^2}+(\nabla \varphi_i,\nabla \varphi_j)_{L^2}=\delta_{ij}.
$$
Then the function $\rho(s):=\sum_{j=1}^n|\varphi_j(s)|^2$ satisfies the
inequality
$$
\|\rho\|_{L^2}\le\frac1{2\sqrt{\pi}}m^{-1}{n^{1/2}}.
$$
\end{theorem}

\begin{remark}
{\rm
The difference in the formulations of Theorems~\ref{Th:A-Lieb-S2} and
\ref{Th:A-Lieb-S2-scal}
is due to the fact that the vector Laplacian
is positive $-\mathbf{\Delta}\ge2I$,
while the scalar Laplacian $-\Delta\ge0$ is just non-negative.
}
\end{remark}

In terms of the scalar product \eqref{alpha-orth}
estimate~\eqref{LiebS2} is as follows.

\begin{corollary}\label{Cor:mtoalpha2} Under the assumptions
of Theorem \ref{Th:A-Lieb-S2}  let the family $\{v_j\}_{j=1}^n$,
$\operatorname{div}v_j=0$, be
 orthonormal  with respect~to
\begin{equation}\label{alpha-orth}
(v_i,v_j)_{L^2}+\alpha
(\rot v_i,\rot v_j)_{L^2}=\delta_{ij}.
\end{equation}
Then the function $\rho(s)=\sum_{j=1}^n|v_j(s)|^2$ satisfies
%$$
\begin{equation}\label{S2alpha}
\aligned
\|\rho\|_{L^2}\le\frac1{2\sqrt{\pi}}\frac{n^{1/2}}{\alpha^{1/2}}\,.
\endaligned
\end{equation}
%$$
\end{corollary}
%Indeed, this statement follows from Theorem~\ref{Th:A-Lieb-S2} by the proper scaling.

\begin{proposition}\label{Prop:seriesS2}
The following inequality holds for   $m>0$
\begin{equation}
\label{ineqS2}
F(m):=m^2\sum_{n=1}^\infty\frac {2n+1}{\bigl(m^2+n(n+1)\bigr)^2}<1.
\end{equation}
\end{proposition}
\begin{proof}
We write $F(m)$ as follows
$$
F(m)=\frac1{m^2}\sum_{n=1}^\infty (2n+1)f\left(\frac{n(n+1)}{m^2}\right),
\quad f(t)=\frac1{(1+t)^2}\,.
$$
%where
%$$
%f(t)=\frac1{(1+t)^2}\,.
%$$
The following asymptotic expansion holds for $F(m)$% as $m\to\infty$
(see \cite[Lemma~3.5]{IZ})
$$
F(m)=\int_0^\infty f(t)dt-\frac1{m^2}\frac23f(0)
+O(m^{-4}), \quad\text{as} \ m\to\infty.
$$
which in view of $f(0)=1$, and $\int_0^\infty f(x)dx=1$ gives that
$$
F(m)=1-\frac1{m^2}\frac23+O(m^{-4}).
$$
This shows that~\eqref{ineqS2} holds for all $m\in[m_0,\infty)$,
where $m_0$ is sufficiently large. A general method for proving
this type of inequalities is to somehow specify the value of $m_0$
and show that~\eqref{ineqS2} also holds on $m\in[0,m_0]$ by a
reliable computer calculation. However, in this specific case we
can prove inequality~\eqref{ineqS2} completely rigorously, and for
this purpose we use a refinement of the method proposed in
\cite{I99JLMS,I12JST}.

Let
$$
a_1=0,\
a_n=a_n(m):=\frac {(n-1)n}{m^2},\ \ n=2,\dots\,.
$$
Then
$$
\aligned
m^{-2}\sum_{n=1}^\infty n f(n(n+1)/m^2)
&=\frac 12
\sum_{n=1}^\infty f(a_{n+1})(a_{n+1}-a_n),\\
m^{-2}\sum_{n=1}^\infty(n+1)f(n(n+1)/m^2)
&=\frac 12
\sum_{n=1}^\infty f(a_{n+1})(a_{n+2}-a_{n+1}).
\endaligned
$$
Therefore
$$
F(m)=
\frac 12 f(a_2)(a_2-a_1)+\sum\limits_{n=2}^\infty
\frac {f(a_n)+f(a_{n+1})}2\,(a_{n+1}-a_n),
$$
and inequality~\eqref{ineqS2} is equivalent to
$$
\aligned
\frac 12 f(a_2)(a_2-a_1)+\sum\limits_{n=2}^\infty
\frac {f(a_n)+f(a_{n+1})}2\,(a_{n+1}-a_n)<\\<
\sum_{n=1}^\infty\int_{a_n}^{a_{n+1}}f(x)dx=
\int_{a_1}^{a_{2}}f(x)dx+\sum_{n=2}^\infty\int_{a_n}^{a_{n+1}}f(x)dx=1,
\endaligned
$$
or, equivalently, to
$$
\aligned
\int_{a_1}^{a_{2}}f(x)dx-
\frac 12 f(a_2)(a_2-a_1)>\\>\sum\limits_{n=2}^\infty
\left(
\frac {f(a_n)+f(a_{n+1})}2\,(a_{n+1}-a_n)-
\int_{a_n}^{a_{n+1}}f(x)dx\right)=:\sum_{n=2}^\infty R_n(m).
\endaligned
$$
Next, for $f(x)=1/(1+x)^2$ we have
$$
\int_{a_1}^{a_{2}}f(x)dx-
\frac 12 f(a_2)(a_2-a_1)=\frac{m^2+4}{(m^2+2)^2}
$$
and
$$
\frac {f(a)+f(b)}2\,(b-a)-
\int_{a}^{b}f(x)dx=
\frac12\frac{(b-a)^3}{(1+a)^2(1+b)^2}\,.
$$
Therefore
$$
R_n(m)=\frac4{m^2}\frac1m
\frac{( n/m)^3}{\bigl((1+n(n-1)/m^2\bigr)^2\bigl(1+n(n+1)/m^2\bigr)^2},
$$
and we see that inequality \eqref{ineqS2} holds if and only if
$$
\aligned
m^2\left[\frac{m^2+4}{(m^2+2)^2}+R_1(m)\right]=m^2\left[\frac{m^2+4}{(m^2+2)^2}+\frac4{m^2}\frac1{(m^2+2)^2}\right]=1>\\>
4\frac1m\sum_{n=1}^\infty
\frac{( n/m)^3}{\bigl((1+n(n-1)/m^2\bigr)^2\bigl(1+n(n+1)/m^2\bigr)^2},
\endaligned
$$
that is,
$$
R(m):=
\frac1m\sum_{n=1}^\infty
\frac{( n/m)^3}{\bigl((1+n(n-1)/m^2\bigr)^2\bigl(1+n(n+1)/m^2\bigr)^2}<\frac14.
$$
Next, we find a lower bound for the denominator in terms of a
function depending on $n/m$.  Taking into account  that
$$
F'(m)=2m\sum_{n=1}^\infty\frac{(2n+1)(n(n+1)-m^2)}{\bigl(m^2+n(n+1)\bigr)^3}>0
$$
for $m\in[0,\sqrt{2}]$, we see that it suffices to prove
inequality~\eqref{ineqS2} for $m\in[\sqrt{2},\infty)$.

The largest constant $k$ in the inequality
$$
\frac{\bigl((m^2+n(n-1)\bigr)^2\bigl(m^2+n(n+1)\bigr)^2}{m^8}\ge k\left(1+\frac{n^2}{m^2}\right)^4
$$
holding for $m^2\ge2$ and $n\in\mathbb{N}$ satisfies
$$
\sqrt{k}=1-\max_{m^2\ge2,\ n\in\mathbb{N}}\frac{n^2}{(m^2+n^2)^2}=
1-\max_{n\in\mathbb{N}}\frac{n^2}{(2+n^2)^2}=\frac89.
$$
This gives that for $m\ge\sqrt{2}$
$$
R(m)<\frac{81}{64}\frac1m\sum_{n=1}^\infty g(n/m),\quad g(x)=\frac{x^3}{(1+x^2)^4}\,.
$$

 The function $g(x)$ has a  global maximum  at
$x_0=\left(\frac3{5}\right)^{1/2}$ and is decreasing for $x>x_0$. Therefore it is geometrically
clear that for all $m>0$
$$
\aligned
 \frac1m\sum_{n=1}^\infty
g(n/m)< x_0g(x_0)+\int_{x_0}^\infty g(x)dx=
\frac{225}{4096}+\frac{175}{3072}=\frac{1375}{12288}.
\endaligned
$$
The proof is now complete, since  for $m\ge\sqrt{2}$
$$
R(m)<\frac{81}{64}\frac1m\sum_{n=1}^\infty g(n/m)<\frac{81}{64}\frac{1375}{12288}=\frac{37125}{262144}=0.14162<\frac14.
$$
\end{proof}

%\begin{remark}
%{\rm
% It is interesting to note that
% $$
% R(0)=\frac14.
% $$
% }
%\end{remark}


\begin{thebibliography}{99}

 \bibitem{Araki}
H. Araki,
\textrm{On an inequality of Lieb and Thirring}.
\emph{Lett.
Math. Phys.}, \textbf{19}:2 (1990), 167--170.

\bibitem{B-V}
A. Babin and M. Vishik,
 \emph{Attractors of Evolution Equations.}
 Studies in Mathematics and its Applications, vol 25.
 North-Holland Publishing Co., Amsterdam, 1992.

 \bibitem
%[BFR80]
{BFR80}
 J. Bardina, J. Ferziger, and  W. Reynolds,
  \emph{Improved subgrid scale models for large eddy simulation},
  in Proceedings of the 13th AIAA Conference on Fluid and Plasma Dynamics, (1980).
%---------------------------------------------------------------
\bibitem{Bardina}
Y. Cao,  E. M. Lunasin, and E.S. Titi,
\textrm{Global well-posedness of the three-dimensional
viscous and inviscid simplified Bardina turbulence models.}
\textit{Commun. Math. Sci.} \textbf{4}:4 (2006),  823--848.

\bibitem{Ch-I2001}
V. V. Chepyzhov and A. A. Ilyin,
\textrm{A note on the fractal dimension of attractors
of dissipative dynamical systems}.
\textit{Nonlinear Anal.}
\textbf{44} (2001), 811--819.


\bibitem{Ch-I}
V. V. Chepyzhov and A. A. Ilyin,
\textrm{On the fractal dimension of invariant sets; applications
 to Navier--Stokes equations}.
\textit{Discrete Contin. Dyn. Syst.}
\textbf{10}: 1-2 (2004),
 117--135.


\bibitem{Dubrovin}
B.A. Dubrovin, S.P. Novikov,  and A.T. Fomenko,
\textit{Modern Geometry. Methods and Applications: Part I}.
Nauka, Moscow, 1979; Eglish translation,
Graduate texts in mathematics, 93, Springer, New York, 1984.


\bibitem{I90}
\textrm{A.A. Ilyin,}
\textrm{The Navier--Stokes and Euler equations on two dimensional
closed manifolds.}
\textit{\ Mat. Sbornik} \textbf{181}:4 (1990), 521--539;
 \textrm{English transl. in} \textit{\ Mathematics of the USSR-Sbornik}
\textbf{69}:2 (1991).


 \bibitem{I99JLMS}\textrm{A.A. Ilyin,}
\textrm{Best constants in Sobolev inequalities on the sphere and in
Euclidean space.} \textit{ J. London Math. Soc.}(2) \textbf{59}
(1999), 263-286.

\bibitem{IMT} A. A. Ilyin, A. Miranville, and E. S. Titi,
\textrm{Small viscosity
sharp estimates for the global attractor of the 2-D damped-driven
Navier-Stokes equations}.
\textit{Commun. Math. Sci.}
\textbf{2} (2004), 403--426.

 \bibitem{I12JST}\textrm{A.A.~Ilyin,}
\textrm{ Lieb--Thirring inequalities on some manifolds.}
 {\it J. Spectr. Theory} \textbf{2}  (2012), 57--78.

 %\bibitem{ZIL-JFA}
%\textrm{A.A.\,Ilyin, A.A.\,Laptev, and S.V.\,Zelik,}
%\textrm{Lieb--Thirring constant on the sphere and on the torus.}
%\textit{J. Func. Anal.} 279 (2020), 108784.

\bibitem{Lap70}
\textrm{A.A.~Ilyin and  S.V.~Zelik,}
\textrm{Sharp dimension estimates of the
    attractor of the damped 2D Euler--Bardina equations.}
In book: \textit{EMS Series of Congress Reports Vol. 18. Partial Differential Equations, Spectral Theory,
    and Mathematical Physics}, EMS Press, Berlin, 2021, p. 209--229.


\bibitem{arxiv}
\textrm{A.A.~Ilyin, A.G.~Kostianko, and  S.V.~Zelik,}
\textrm{Sharp upper and lower bounds of the
attractor dimension for   3D damped Euler--Bardina equations}.
arXiv: 2106.09077.

\bibitem{Titi-Varga}
V. K. Kalantarov and E. S. Titi, {Global attractors and determining modes for the 3D
Navier--Stokes--Voight equations},
\textit{Chin. Ann. Math.}
\textbf{30}:6 (2009), 697--714.

\bibitem{T}
R.~Temam,
\emph{Infinite Dimensional Dynamical Systems in
Mechanics and Physics, \rm 2nd ed.}
Sprin\-ger-Ver\-lag, New York  1997.

\bibitem{IZ}\textrm{S.V.~Zelik and A.A.~Ilyin,} \textrm{Green's
    function asymptotics and sharp  interpolation inequalities.}
    \textit{ Uspekhi Mat. Nauk} \textbf{69}:2 (2014), 23--76;
 English transl. in
 \textit{Russian Math. Surveys} 69:2 (2014).

 \bibitem{LiebJFA}
E. H. Lieb, {An $L^p$ bound for the
    Riesz and Bessel potentials of orthonormal functions},
\textit{J. Func. Anal.}
\textbf{51} (1983),  159--165.


\bibitem{LT} \textrm{E.~Lieb and W.~Thirring,} \textrm{
    Inequalities for the moments of the eigenvalues of the
    Schr\"o\-dinger Hamiltonian and their relation to Sobolev
    inequalities, Studies in Mathematical Physics. Essays in honor
    of Valentine Bargmann,} \textrm{ Princeton University Press},
 \textrm{ Princeton NJ}, 269--303 (1976).

 \bibitem{traceSimon}
B. Simon, \emph{Trace Ideals and Their Applications, \rm 2nd ed.}
Amer. Math. Soc., Providence RI, 2005.

 \end{thebibliography}
\end{document}